\documentclass[a4paper,reqno,11pt]{amsart}
\usepackage{amsmath}
\usepackage{amssymb}
\usepackage{cases}
\usepackage{enumitem}
\allowdisplaybreaks[4]

\newtheorem{thm}{Theorem}[section]
\newtheorem{prop}[thm]{Proposition}

\newtheorem{lem}[thm]{Lemma}
\newtheorem{q}[thm]{Question}

\newtheorem{claim}[thm]{Claim}

\theoremstyle{definition}

\theoremstyle{remark}
\newtheorem{remark}[thm]{Remark}

\numberwithin{equation}{section}

\newcommand{\bQ}{\mathbb{Q}}

\newcommand\OO{{\mathcal{O}}}

\newcommand{\Proj}{\operatorname{Proj}}

\usepackage{todonotes}

\begin{document}

\title{Characterizing terminal Fano threefolds with the smallest anti-canonical volume}
\date{\today}
\author{Chen Jiang}
\address{Chen Jiang, Shanghai Center for Mathematical Sciences, Fudan University, Jiangwan Campus, Shanghai, 200438, China}
\email{chenjiang@fudan.edu.cn}

\begin{abstract}
It was proved by J.~A.~Chen and M.~Chen that a terminal Fano $3$-fold $X$ satisfies  $(-K_X)^3\geq \frac{1}{330}$.
We show that a non-rational $\bQ$-factorial terminal Fano $3$-fold $X$ with $\rho(X)=1$ and $(-K_X)^3=\frac{1}{330}$ is a weighted hypersurface of degree $66$ in $\mathbb{P}(1,5,6,22,33)$.
\end{abstract}

\maketitle
\pagestyle{myheadings} \markboth{\hfill C.~Jiang \hfill}{\hfill Terminal Fano $3$-folds with the smallest anti-canonical volume\hfill}


\section{Introduction}
Throughout this paper, we work over the field of complex numbers $\mathbb{C}$.

A normal projective variety $X$ is called a {\it Fano variety} (or {\it $\bQ$-Fano variety} in some literature) if the anti-canonical divisor $-K_X$ is ample. 
A {\it terminal} Fano variety is a Fano variety with at worst terminal singularities. 

According to the minimal model program, Fano varieties form a fundamental class among research objects of birational geometry. Motivated by the classification theory of $3$-dimensional algebraic varieties, we are interested in the study of explicit geometry of terminal Fano $3$-folds.

Given a terminal Fano $3$-fold $X$, it was proved in \cite[Theorem~1.1]{CC08} that $(-K_{X})^3\geq \frac{1}{330}$. This lower bound is optimal, because it is attained when $X=X_{66}\subset \mathbb{P}(1,5,6,22,33)$ is a general weighted hypersurface of degree $66$. Moreover, \cite[Theorem~1.1]{CC08} showed that when $(-K_{X})^3= \frac{1}{330}$, then $X$ has exactly the same virtual orbifold singularities and the same Hilbert polynomial as $X_{66}$ (see Proposition~\ref{prop X=X66}).
So it is interesting to ask the following question:
\begin{q}\label{main q}
Let $X$ be a terminal Fano $3$-fold with $(-K_X)^3=\frac{1}{330}$. 
Is $X$ (a $\mathbb{Q}$-Gorenstein deformation of) a quasi-smooth weighted hypersurface of degree $66$ in 
 $\mathbb{P}(1,5,6,22,33)$?
 \end{q}

Note that every quasi-smooth weighted hypersurface $X_{66}$ of degree $66$ in 
 $\mathbb{P}(1,5,6,22,33)$ is always  a $\bQ$-factorial terminal Fano $3$-fold with $\rho(X)=1$ by \cite{Dol82, IF00} and is birationally rigid (and in particular, non-rational) by \cite{CP17}. So the main goal of this note is to give an affirmative answer to Question~\ref{main q} in the category of  non-rational $\bQ$-factorial terminal Fano $3$-folds with $\rho=1$.

\begin{thm}\label{mainthm}
Let $X$ be a $\bQ$-factorial terminal Fano $3$-fold with $\rho(X)=1$ and $(-K_X)^3=\frac{1}{330}$. Suppose that $X$ is non-rational. Then $X$ is a weighted hypersurface of degree $66$ in $\mathbb{P}(1,5,6,22,33)$ defined by a weighted homogeneous polynomial $F$ of degree $66$, where 
$$F(x, y, z, w, t)=t^2+F_0(x, y, z, w)$$ in suitable homogeneous coordinates $[x: y: z: w: t]$ of $\mathbb{P}(1, 5, 6, 22, 33)$.
\end{thm}
The idea of the proof is as the following: 
as $X\simeq \Proj R(X, -K_X)$ where $\Proj R(X, -K_X)$ is the section ring of $-K_X$, 
it suffices to show that $R(X, -K_X)\simeq R(Y, \mathcal{O}_Y({1}))$ for $Y$ a weighted hypersurface of degree $66$ in $\mathbb{P}(1,5,6,22,33)$.
By \cite[Theorem~1.1]{CC08} (see Proposition~\ref{prop X=X66}), these two graded $\mathbb{C}$-algebras have the same dimensions on homogeneous parts, so the goal is to determine generators of $R(X, -K_X)$ and their relations. 
The key ingredient is the special geometry of anti-pluri-canonical systems of $\bQ$-factorial terminal Fano $3$-fold with $\rho=1$ proved in \cite{CJ16} (see Lemma~\ref{cor irr}) which was first observed by Alexeev \cite{Ale94} for anti-canonical systems, where the assumption that $X$ is $\bQ$-factorial terminal with $\rho=1$ is essentially used.

 \begin{remark}\label{main remark}
 The method in this note can be used to characterize other  terminal Fano $3$-folds. For example, by the same method, it could be shown that, if $X$ is a non-rational $\bQ$-factorial terminal Fano $3$-fold with $\rho(X)=1$ such that there exists a
  general weighted hypersurface $${X_{6d}}\subset \mathbb{P}(1,a,b,2d,3d)$$ of degree $6d$ as in \cite[List~16.6, Table~5, No.~14, No.~34, No.~53, No.~70, No.~72, No.~82, No.~88--90, No.~92, No.~94]{IF00} with  $(-K_X)^3=(-K_{{X_{6d}}})^3$ and $B_X=B_{{X_{6d}}}$ (see Section~\ref{sec 2}), then  $X$ is a weighted hypersurface of degree $6d$ in $\mathbb{P}(1,a,b,2d,3d)$. Similar characterization could also be done for some other examples in \cite[List~16.6, Table~5]{IF00} or possibly even for weighted complete intersections of lower codimensions. However those examples are less interesting than  terminal Fano $3$-folds with the smallest anti-canonical volume, so the author is not so motivated to write the proof down and check full details. Maybe it is better to leave it as  an exercise for students.\end{remark}

%
\section{Reid's Riemann--Roch formula}\label{sec 2}

A {\it basket} $B$ is a collection of pairs of integers (permitting
weights), say $\{(b_i,r_i)\mid i=1, \cdots, s; b_i\ \text{is coprime
 to}\ r_i\}$. 

Let $X$ be a terminal Fano $3$-fold. According to 
Reid \cite{YPG}, there is  a basket of orbifold points (called {\it Reid basket})
$$B_X=\bigg\{(b_i,r_i)\mid i=1,\cdots, s; 0<b_i\leq \frac{r_i}{2};b_i \text{ is coprime to } r_i\bigg\}$$
associated to $X$, where a pair $(b_i,r_i)$ corresponds to a (virtual) orbifold point $Q_i$ of type $\frac{1}{r_i}(1,-1,b_i)$. 
Moreover, by Reid's Riemann--Roch formula and the Kawamata--Viehweg vanishing theorem, for any positive integer $m$, 
\begin{align*}
h^0(X, -mK_X)={}&\chi(X, \OO_X(-mK_X))\\={}&\frac{1}{12}m(m+1)(2m+1)(-K_X)^3+(2m+1)-l(m+1)
\end{align*}
where
$l(m+1)=\sum_i\sum_{j=1}^m\frac{\overline{jb_i}(r_i-\overline{jb_i})}{2r_i}$ and the first sum runs over all orbifold points in Reid basket (\cite[2.2]{CJ16}). Here $\overline{jb_i}$ means the smallest non-negative residue of $jb_i \bmod r_i$.


\begin{prop}\label{prop X=X66}
Let $X$ be a terminal Fano $3$-fold with $(-K_X)^3=\frac{1}{330}$. 
Then $B_X=\{(1,2), (2,5), (1, 3), (2, 11)\}$. Moreover,
$$
\sum_{m\geq 0}h^0(X, -mK_X)q^m=\frac{1-q^{66}}{(1-q)(1-q^5)(1-q^6)(1-q^{22})(1-q^{33})}.
$$
 \end{prop}
\begin{proof}
The characterization of $B_X$ is given in \cite[Theorem~1.1(iii)]{CC08}.
 For a general weighted hypersurface $${X_{66}}\subset \mathbb{P}(1,5,6,22,33)$$ of degree $66$, $(-K_X)^3=(-K_{{X_{66}}})^3=\frac{1}{330}$ and $B_X=B_{{X_{66}}}$ (\cite[List~16.6, Table~5, No.~95]{IF00}). By Reid's Riemann--Roch formula, $h^0(X, -mK_X)$ depends only on $(-K_X)^3$ and $B_X$. Note that $\mathcal{O}_{{X_{66}}}(-K_{{X_{66}}})=\mathcal{O}(1)|_{{X_{66}}}$.
So 
\begin{align*}
\sum_{m\geq 0}h^0(X, -mK_X)q^m={}&\sum_{m\geq 0}h^0({X_{66}},-mK_{{X_{66}}} )q^m\\
={}&\frac{1-q^{66}}{(1-q)(1-q^5)(1-q^6)(1-q^{22})(1-q^{33})}
\end{align*}
by \cite[Theorem~3.4.4]{Dol82}.
\end{proof}

 \section{Proof of Theorem~\ref{mainthm}}
 


 \begin{lem}\label{cor irr}
Let $X$ be a $\bQ$-factorial terminal Fano $3$-fold with $\rho(X)=1$ and $(-K_X)^3=\frac{1}{330}$. Then for 
\begin{enumerate}
    \item $h^0(X, -K_X)=1$ and the unique element in $|-K_X|$ is a prime divisor;
    \item $|-5K_X|$ is composed with an irreducible pencil of surfaces;
    \item $|-6K_X|$ is not composed with a pencil of surfaces.
\end{enumerate}
\end{lem}
\begin{proof}
Recall that by Proposition~\ref{prop X=X66},
$$h^0(X, -mK_X)=\begin{cases}1 & \text{if }1\leq m\leq 4; \\2 & \text{if } m=5;\\3  & \text{if } m=6.\end{cases}$$
(1) is a direct consequence of  \cite[Theorem~3.2]{CJ16} for $m=1$ (or \cite[Theorem~2.18]{Ale94}),
(2) is a direct consequence of the fact that $h^0(X, -5K_X)=2$, and (3) is a direct consequence of \cite[Theorem~3.4]{CJ16}.
\end{proof}

 \begin{remark}
We do not know whether Lemma~\ref{cor irr}(2)(3) remains true or not if we only assume that $X$ is a terminal Fano $3$-fold. The current proof essentially relies on the assumption that $X$ is $\mathbb{Q}$-factorial with $\rho(X)=1$ as in \cite[Theorem~3.2]{CJ16}  or \cite[Theorem~2.18]{Ale94}.
If one can drop these conditions in  Lemma~\ref{cor irr}(2)(3), then one can drop these conditions in
Theorem~\ref{mainthm}.
 \end{remark}
 \begin{proof}[Proof of Theorem~\ref{mainthm}]
 Recall that for a Weil divisor $D$ on $X$, $$H^0(X, D)=\{f\in \mathbb{C}(X)^{\times}\mid \text{div}(f)+D\geq 0\}\cup \{0\}$$
 can be viewed as a $\mathbb{C}$-linear subspace of the function field $\mathbb{C}(X)$.
For $m\in \{1,5,6,22,33\}$, take $f_m\in H^0(X, -mK_X)\setminus \{0\}$ to be a general element.
We can define $3$ rational maps by these functions:
\begin{align*}
\Phi_{6}: {}&X\dashrightarrow \mathbb{P}(1, 5, 6); \\
{}& P\mapsto [f_1(P):f_5(P):f_6(P)];\\
\Phi_{22}: {}&X\dashrightarrow \mathbb{P}(1, 5, 6, 22); \\
{}& P\mapsto [f_1(P):f_5(P):f_6(P):f_{22}(P)];\\
\Phi_{33}: {}&X\dashrightarrow \mathbb{P}(1, 5, 6, 22, 33);\\
{}& P\mapsto [f_1(P):f_5(P):f_6(P):f_{22}(P): f_{33}(P)].
\end{align*}
We claim that they have the following geometric properties.
\begin{claim}
\begin{enumerate}

\item $\Phi_{6}$ is dominant;
\item $\Phi_{22}$ is generically finite of degree $2$;
\item $\Phi_{33}$ is birational onto its image;
\item let $Y$ be the closure of $\Phi_{33}(X)$ in $\mathbb{P}(1, 5, 6, 22, 33)$, then $Y$ is defined by a weighted homogeneous polynomial $F$ of degree $66$, where 
$$F(x, y, z, w, t)=t^2+F_0(x, y, z, w)$$ in suitable homogeneous coordinates $[x: y: z: w: t]$ of $\mathbb{P}(1, 5, 6, 22, 33)$.
\end{enumerate}\end{claim}
\begin{proof}
(1) By Lemma~\ref{cor irr}, $|-5K_X|$ is composed with an irreducible pencil of surfaces and 
$|-6K_X|$ is not composed with a pencil of surfaces. 
Recall that $h^0(X, -6K_X)=3$, so 
$H^0(X, -6K_X)$ is spanned by $\{f_1^6, f_1f_5, f_6\}$. 
Hence $\Phi_6$ is birational to the rational map $X\dashrightarrow \mathbb{P}^2$ defined by $|-6K_X|$, which is obviously dominant.

\medskip

(2) By \cite[Theorem~4.4.11]{Phd}  (taking $m_0=\mu_0=5$ and $m_1=6$),
we conclude that $|-22K_X|$ defines a generically finite map onto its image. Hence a general $f_{22}$ is not constant along general fibers of $\Phi_6$. Therefore $\Phi_{22}$ is generically finite. 
To compute the degree of $\Phi_{22}$, 
take  a resolution
$\pi: W\to X$ such that for $m\in \{5,6,22\}$,
$\pi^*(-mK_X)=M_m+F_m$ where $M_m$ is free and $F_m$ is the fixed part.
Then
$$
\deg \Phi_{22}=(M_5\cdot M_6\cdot M_{22})\leq (\pi^*(-5K_X)\cdot \pi^*(-6K_X)\cdot \pi^*(-22K_X))=2.
$$
As $X$ is non-rational, we conclude that $\deg \Phi_{22}=2$.

\medskip

(3) By \cite[Theorem~5.11]{CJ16} (taking $m_0=\mu_0=5$ and $m_1=6$),
we conclude that $|-33K_X|$ defines a birational map onto its image. As $f_{33}$ is general, it can separate two points in general fibers of $\Phi_{22}$, so  $\Phi_{33}$ is birational onto its image.

\medskip

(4) Note that $h^0(X, -66K_X)=172$ by Reid's Riemann--Roch formula (see Proposition~\ref{prop X=X66}). On the other hand, the equation
$$n_1+5n_2+6n_3+22n_4+33n_5=66$$ has exactly $173$ solutions in $\mathbb{Z}_{\geq 0}^5$. So there exists a weighted homogeneous polynomial $F(x, y, z, w, t)$ of degree $66$ with $\text{wt}(x, y, z, w, t)=(1,5,6,22,33)$ such that 
$$F(f_1, f_5, f_6, f_{22}, f_{33})=0.$$
So $Y$ is contained in $(F=0)\subset \mathbb{P}(1, 5, 6, 22, 33)$. 

We claim that $Y=(F=0)$ and $t^2$ has no-zero coefficient in $F$.
Otherwise, $Y$ is defined by a weighted homogeneous polynomial $\tilde{F}$  of degree $\leq 66$  of the form 
$$\tilde{F}(x, y, z, w, t)=t\tilde{F}_1(x, y, z, w)+\tilde{F}_2(x, y, z, w).$$
But then $Y$ is birational to $\mathbb{P}(1, 5, 6, 22)$ under the rational projection map
\begin{align*}
{}&\mathbb{P}(1, 5, 6, 22, 33)\dashrightarrow \mathbb{P}(1, 5, 6, 22);\\
{}&[x:y:z:w:t]\mapsto [x:y:z:w].
\end{align*}
This contradicts to the assumption that $X$ is non-rational. So  $Y=(F=0)$ and $t^2$ has no-zero coefficient in $F$. After a suitable coordinate change we may assume that
$F=t^2+F_0(x, y, z, w)$.
\end{proof}
By the above claim, $F$ is the only relation on $f_1, f_5, f_6, f_{22}, f_{33}$. Denote $\mathcal{R}$ to be the  graded sub-$\mathbb{C}$-algebra of $$R(X, -K_X)=\bigoplus_{m\geq 0}H^0(X, -mK_X)$$
generated by $\{f_1, f_5, f_6, f_{22}, f_{33}\}$. Then we have a natural isomorphism between graded $\mathbb{C}$-algebras
$$
\mathcal{R}\simeq \mathbb{C}[x, y, z, w, t]/(t^2+F_0)
$$
by sending $f_1\mapsto x$, $f_5\mapsto y$,  $f_6\mapsto z$, $f_{22}\mapsto w$, $f_{33}\mapsto t$.
Write $\mathcal{R}=\bigoplus_{m\geq 0}\mathcal{R}_m$ where $\mathcal{R}_m$ is the homogeneous part of degree $m$. Then
by \cite[3.4.2]{Dol82}, 
$$
\sum_{m\geq 0}\dim_\mathbb{C} \mathcal{R}_m q^m=\frac{1-q^{66}}{(1-q)(1-q^5)(1-q^6)(1-q^{22})(1-q^{33})}.
$$
So by Proposition~\ref{prop X=X66}, $\mathcal{R}_m=H^0(X, -mK_X)$ for any $m\in \mathbb{Z}_{\geq 0}$, and hence the inclusion $\mathcal{R}\subset R(X, -K_X)$ is an isomorphism.
This implies that 
$$X\simeq \Proj R(X, -K_X) \simeq  \Proj\mathcal{R}\simeq Y. $$
This finishes the proof.
 \end{proof}

\section*{Acknowledgments}
The author was supported by National Key Research and Development Program of China (Grant No.~2020YFA0713200).  This paper was written during the author's visit to Xiamen University in July 2021, and he is grateful for the hospitality and support of Wenfei Liu and XMU. The author would like to thank Yifei Chen for useful comments and discussions.

\end{document}